\documentclass[11pt,letterpaper,reqno]{amsart}
\usepackage{fullpage}
\usepackage{amsmath,amsthm,amsfonts,amssymb,amscd}
\usepackage{empheq}
\usepackage{caption}
\usepackage{lipsum}
\usepackage{lastpage}
\usepackage{thmtools}
\usepackage{enumerate}
\usepackage[shortlabels]{enumitem}
\usepackage{fancyhdr}
\usepackage{mathrsfs}
\usepackage{xcolor}
\usepackage{graphicx}
\usepackage{listings}
\usepackage{bbm}
\usepackage{array}
\usepackage{hyperref}
\usepackage[makeroom]{cancel}
\usepackage[utf8]{inputenc}
\usepackage{comment}
\usepackage{makecell}

\usepackage{bbm}
\usepackage{stmaryrd}

\mathtoolsset{showonlyrefs=true}
\numberwithin{figure}{section}
\numberwithin{table}{section}

\hypersetup{%
  colorlinks=true,
  linkcolor=blue,
  citecolor=blue,
  linkbordercolor={0 0 1}
}

\allowdisplaybreaks

\newtheorem{theorem}{Theorem}[section]
\newtheorem{corollary}[theorem]{Corollary}
\newtheorem{lemma}[theorem]{Lemma}

\mathtoolsset{showonlyrefs=true}
\numberwithin{equation}{section}

\newcommand{\lrbb}[1]{\left\llbracket#1\right\rrbracket}

\newcommand{\lrp}[1]{\left(#1\right)}

\newcommand{\condf}{{\mathfrak{f}}}
\newcommand{\midmid}{\parallel}

\newcommand{\nw}{{\textnormal{new}}}

\DeclareMathOperator{\QQ}{\mathbb{Q}}

\DeclareMathOperator{\Tr}{Tr}

\newcommand{\T}{\widehat T}

\setlist[enumerate]{leftmargin=*,widest=0}
\setlist[itemize]{leftmargin=*,widest=0}
\makeatletter
\def\subsection{\@startsection{subsection}{2}%
  \z@{.5\linespacing\@plus.7\linespacing}{.3\linespacing}%
  {\normalfont\bfseries}
}

\begin{document}

\title{Hecke Eigenvalue Equidistribution over the Newspaces with Nebentypus}


\author[E. Ross]{Erick Ross}
\address[E. Ross]{School of Mathematical and Statistical Sciences, Clemson University, Clemson, SC, 29634}
\email{erickjohnross@gmail.com}

\begin{abstract}
    Fix a prime $p$, and let $\widehat T_p^{\mathrm{new}}(N,k,\chi) := \chi(p)^{-1/2} p^{-(k-1)/2} T_p^{\mathrm{new}}(N,k,\chi)$ denote the normalized $p$'th Hecke operator over the newspace with nenbentypus $S_k^{\mathrm{new}}(N,\chi)$. In this paper, we determine the distribution of the eigenvalues of $\widehat T_p^{\mathrm{new}}(N,k,\chi)$ as $N+k \to \infty$.  
\end{abstract}

\maketitle

\section{Introduction} \label{sec:intro}
For $N \ge 1$ and $k\ge 2$ even, let $S_k(N)$ denote the space of cuspidal modular forms of weight $k$ and modular group $\Gamma_0(N)$. Additionally, let $S_k^\nw(N) \subseteq S_k(N)$ denote its new subspace. For primes $p$, let $\T_p(N,k) := p^{-(k-1)/2} \,T_p(N,k)$ and $\T_p^\nw(N,k) := p^{-(k-1)/2} \T^\nw_{p}(N,k)$ denote the $p$'th normalized Hecke operators over $S_k(N)$ and $S_k^\nw(N)$, respectively. 

The asymptotic behavior of the eigenvalues of these Hecke operators has been studied in many different settings. From one asymptotic perspective, one can fix $N$ and $k$, and ask about the behavior of the eigenvalues of $\T^\nw_p(N,k)$ as $p \to \infty$. Here, the eigenvalues of $\T^\nw_p(N,k)$ correspond with the normalized Fourier coefficients $\widehat a_f(p)$ of newforms $f = \sum_{m\ge1} m^{(k-1)/2} \widehat a_f(m) q^m \in S_k^\nw(N)$. And the asymptotic behavior of these Fourier coefficients has been an important area of research in Number Theory for many years. The Sato-Tate conjecture (now a theorem \cite{BGHT}), for example, states that the prime-indexed Fourier coefficients $\widehat a_f(p)$ of a non-CM newform $f$ tend to the Sato-Tate distribution $\mu_\infty$ as $p \to \infty$. Similarly, it is known that the prime-indexed Fourier coefficients $\widehat a_f(p)$ of a CM newform tend to the CM distribution $\mu_{\text{CM}}$ as $p \to \infty$ \cite[Theorem 15.4]{harmonic-maass-forms}.

From a slightly different perspective, one can fix $p$ and ask about the behavior of the eigenvalues of $\T^\nw_p(N,k)$ as $N+k \to \infty$. From this perspective, Serre showed that the eigenvalues of $\T_p(N,k)$ and $\T_p^\nw(N,k)$ tend to yet another distribution $\mu_p$ as $N+k \to \infty$ \cite[Theorems 1, 2]{serre}. 

A natural question one might then ask is if the same result holds over the spaces with nebentypus $S_k(N,\chi)$ and $S_k^\nw(N,\chi)$. In fact, Serre noted that his result (with the same distribution $\mu_p$) could be extended to the full spaces with nebentypus $S_k(N,\chi)$ via an identical strategy \cite[Theorem 4]{serre}. However, \cite{serre} did not contain the tools to show the corresponding result over the newspaces with nebentypus $S_k^\nw(N,\chi)$.
In this paper, we address this case by determining the distribution of the eigenvalues of $\T_p^\nw(N,k,\chi)$.

One might guess that Serre's $\mu_p$-equidistribution result should still hold over $S_k^\nw(N,\chi)$. However, this turns out to not be true; the eigenvalues of $\T_p^\nw(N,k,\chi)$ are not guaranteed to tend to a distribution (albeit for a trivial reason explained shortly). But when these eigenvalues do tend to an asymptotic distribution, they will tend to $\mu_p$, like for the full space $S_k(N,\chi)$. 

In particular, we show the following result. Here, and throughout this paper, $\condf(\chi)$ denotes the conductor of $\chi$.
\begin{theorem} \label{thm:main-thm}
    Fix a prime $p$, and let $\displaystyle \mu_p(x) := \frac{p+1}{\pi} \frac{(1-x^2/4)^{1/2}}{(p^{1/2} + p^{-1/2})^2-x^2} dx $. Consider $N \ge 1$ coprime to $p$, $k \ge 2$, and $\chi$ Dirichlet characters modulo $N$ such that $\chi(-1) = (-1)^k$. 
    Then assuming it is not the case that $2 \mid \condf(\chi),\ 2 \midmid N/\condf(\chi)$; the eigenvalues of $\T_p^\nw(N,k,\chi)$ are $\mu_p$-equidistributed over $[-2,2]$ as $N+k \to \infty$.
\end{theorem}

We note that when $2 \mid \condf(\chi),\ 2 \midmid N/\condf(\chi)$, the eigenvalues of $\T_p^\nw(N,k,\chi)$ do not tend to a distribution for a trivial reason. In this case, it turns out that $\dim S_k^\nw(N,k,\chi) = 0$ \cite[Proposition 6.1]{ross}, so in fact, there are no eigenvalues of $\T_p^\nw(N,k,\chi)$. 
However, if we exclude the case of $2 \mid \condf(\chi),\ 2 \midmid N/\condf(\chi)$, then it turns out that $\dim S_k^\nw(N,k,\chi) \to \infty$ as $N+k \to \infty$ \cite[Theorem 1.3]{ross}. Hence after excluding this exceptional case, it is perfectly well-defined to ask about the distribution of the eigenvalues of $\T_p(N,k,\chi)$ as $N+k \to \infty$.

Finally, in Section \ref{sec:application}, we will discuss an application of Theorem \ref{thm:main-thm}. In particular, one can use the equidistribution result from Theorem \ref{thm:main-thm} to obtain information about the coefficient fields for newforms in $S_k^\nw(N,\chi)$.

Most of the technical details underlying the techniques used in this paper have already been worked out in other contexts (both by the author, and by others).
In particular, we rely heavily on two key pieces. 
First, Lemma \ref{lemma:serre-characterization} gives a certain characterization of $\mu_p$-equidistribution from \cite{serre} (following the same outline that Serre used to show his original $\mu_p$-equidistribution result).
Second, Lemma \ref{lemma:trace-estimate} gives Hecke operator trace estimates that were computed in \cite{ross} and \cite{cason-et-al}.  

\section{Proof of the Main Theorem}
In this section, we prove the main theorem. Define $\T_m^\nw(N,k,\chi)$ via the normalization $\T_m^\nw(N,k) := \chi(m)^{-1/2} m^{-(k-1)/2} \,T_m(N,k,\chi)$ (taking, for example, the principal branch of $\chi(p)^{-1/2}$ for primes $p$). 
We are using this normalization so that the eigenvalues of $\T_p^\nw(N,k,\chi)$ all lie in the real interval $[-2,2]$.

Now, Theorem \ref{thm:main-thm} claims that the eigenvalues of $\T_p^\nw(N,k,\chi)$ are $\mu_p$-equidistributed over $[-2,2]$ as $N+k \to \infty$. Specifically, this means that 
\begin{align} \label{eqn:equid-def} 
    \frac{1}{\dim S_k^\nw(N,\chi)} \sum_{\lambda \in \mathrm{eigv}\, \T_p(N,k,\chi)} g(\lambda) &\longrightarrow \int_{-2}^2 g(x) \mu_p(x) \quad\text{as}\quad N+k\to \infty
\end{align}
for all $g \in C[-2,2]$ (the real-valued continuous functions over $[-2,2]$).

To prove this result we use the following characterization from Serre.
\begin{lemma}[{\cite[Proposition 2]{serre}}] \label{lemma:serre-characterization}
    To show \eqref{eqn:equid-def} for all $g \in C[-2,2]$, it suffices to just show \eqref{eqn:equid-def} for the $g$ in any sequence $\{g_n\}_{n\ge0}$ of polynomials with $n = \deg g_n$.
\end{lemma}
\begin{proof}
    Assume that \eqref{eqn:equid-def} holds for the polynomials $g \in \{g_n\}_{n\ge0}$.
    First, observe that \eqref{eqn:equid-def} also holds for all polynomials $g$ (since the $\{g_n\}_{n\ge 0}$ form a basis for the set of all polynomials). Then since the set of polynomials is dense in $C[-2,2]$ (via the supremum norm), we furthermore have that \eqref{eqn:equid-def} holds for all $g \in C[-2,2]$, yielding the desired result.
\end{proof}

We will then apply this lemma using the Chebyshev polynomials $X_n(x) := U_n(x/2)$. In particular, to show Theorem \ref{thm:main-thm}, it suffices to show the following identity:
\begin{align} \label{eqn:suff-cond} \tag{$*$}
    \frac{1}{\dim S_k^\nw(N,\chi)} \sum_{\lambda \in \mathrm{eigv}\, \T_p(N,k,\chi)} X_n(\lambda) &\longrightarrow \int_{-2}^2 X_n(x) \mu_p(x) \quad\text{as}\quad N+k\to \infty.
\end{align}

First, we evaluate the RHS of \eqref{eqn:suff-cond}. The value of this integral was already computed in \cite[Equation (20)]{serre}, but we repeat a proof here for convenience of the reader.
\begin{lemma}
    For each $n \ge 0$,
    \begin{align}
        \int_{-2}^2 X_n(x) \mu_p(x) = \mathbbm{1}_{2|n} \,p^{-n/2}.
    \end{align}
\end{lemma}
\begin{proof}
    We make use of the two distributions
    \begin{align}
        \mu_\infty(x) := \frac{1}{\pi}  \lrp{1-x^2/4}^{1/2} dx
        \qquad \text{and} \qquad
        \mu_p(x) := \frac{p+1}{\pi} \frac{(1-x^2/4)^{1/2}}{(p^{1/2} + p^{-1/2})^2-x^2} dx.
    \end{align}
    Now, recall that the generating function for the $X_k(x)$ is given by
    \begin{align}
        \sum_{k = 0}^\infty X_k(x) t^k = \frac{1}{1-tx+t^2}.
    \end{align}
    Substituting $t = p^{-1/2}$ and $t=- p^{-1/2}$ into this identity and summing the two resulting equations then yields
    \begin{align}
        \sum_{k=0}^\infty X_k(x) \mathbbm{1}_{2|k} 2p^{-k/2} &= \frac{1}{1-p^{-1/2}x + p^{-1}} + \frac{1}{1 + p^{-1/2}x + p^{-1}} \\
        &= \frac{2+2p^{-1}}{(1+p^{-1})^2 - p^{-1}x^2} = 2 \frac{p+1}{(p^{1/2} + p^{-1/2})^2 - x^2} = 2 \frac{\mu_p(x)}{\mu_\infty(x)}.
    \end{align}
    Thus
    \begin{align}
        \int_{-2}^2 X_n(x) \mu_p(x) &= \int_{-2}^2 X_n(x) \sum_{k=0}^\infty X_k(x) \mathbbm{1}_{2|k}  p^{-k/2} \mu_\infty(x) \\
        &= \sum_{k=0}^\infty \mathbbm{1}_{2|k}  p^{-k/2}  \int_{-2}^2 X_n(x) X_k(x) \mu_\infty(x) \\
        &= \mathbbm{1}_{2|n}  p^{-n/2},
    \end{align}
    as desired. Note that in the last step, we used the well-known fact that the Chebyshev polynomials are $\mu_\infty$-orthonormal (i.e. that $\int_{-2}^2 X_n(x) X_k(x) \,\mu_\infty(x) = \delta_{n,k}$).
\end{proof}

Next, we evaluate the LHS of \eqref{eqn:suff-cond}.  
Observe that the classical recurrence relation for the $\T_{p^n}$ \cite[Theorem 10.2.9]{cohen-stromberg} is identical to the defining recurrence relation for the Chebyshev polynomials:
\begin{align}
    \T_{p^0} &=\text{Id}, 
    & X_0(x)&=1, \\
    \T_{p^1} &= \T_p, 
    & X_1(x)&=x, \\
    \T_{p^{n+1}} &= \T_p\,\T_{p^n} - \T_{p^{n-1}},  
    & X_{n+1}(x) &= x\, X_n(x) - X_{n-1}(x).
\end{align}
This means that the LHS of \eqref{eqn:suff-cond} is equal to
\begin{align}
    \text{LHS} &= \frac{1}{\dim S_k^\nw(N,\chi)} \sum_{\lambda \in \mathrm{eigv}\, \T_p(N,k,\chi)} X_n(\lambda)  \\
    &= \frac{1}{\dim S_k^\nw(N,\chi)} \sum_{\lambda \in \mathrm{eigv}\, \T_{p^n}(N,k,\chi)} \lambda \\
    &= \frac{\Tr \T^\nw_{p^n}(N,k,\chi) }{\Tr \T_1^\nw(N,k,\chi)}.
\end{align}

Hence, to verify the identity \eqref{eqn:suff-cond}, it suffices to show that
\begin{align}
    \frac{\Tr \T^\nw_{p^n}(N,k,\chi) }{\Tr \T_1^\nw(N,k,\chi)} \longrightarrow \mathbbm{1}_{2|n} p^{-n/2} \quad \text{as} \quad N+k \to \infty. \label{eqn:quot-traces} \tag{$**$}
\end{align}

Asymptotic estimates on these traces were recently computed in \cite{ross} (calculating the growth of the main term) and \cite{cason-et-al} (bounding the error terms). 
In particular, we have the following estimate (where the big-$O$ notation is with respect to both $N$ and $k$).
\begin{lemma}[{\cite[Lemma 6.2]{cason-et-al}, \cite[Equation (6.2)]{ross}}] \label{lemma:trace-estimate}
Let $\omega(N)$ denote the number of distinct prime divisors of $N$ and let $f = \condf(\chi)$. 
Then
\begin{align}
    \Tr \T_m^\nw(N,k,\chi) &= \frac{\mathbbm{1}_{m=\square}}{\sqrt m} \frac{k-1}{12} \psi_{f}^\nw(N) + O(N^{1/2+\varepsilon}), \label{eqn:trace-estimate}
\end{align}
where $\psi_{f}^\nw(N) \ge N / 4^{\omega(N)}$ whenever it is not the case that $2 \mid f,\ 2 \midmid N/f$.
\end{lemma}
\begin{proof}
Equation \eqref{eqn:trace-estimate} comes directly from \cite[Lemma 6.2]{cason-et-al}. Here, we give the details from \cite[Equation (6.2)]{ross} regarding the lower bound on $\psi_{f}^\nw(N)$.
This function is defined as $\psi_{f}^\nw(N) := \psi(f) \cdot \beta{*}\psi_f(N/f)$, where 
$\psi$ and $\beta{*}\psi_f$ are the multiplicative functions defined on prime powers via
\begin{align}
    \psi(p^r) &= p^r \lrp{1+\frac1p}, \\
    \beta{*}\psi_f(p^r) &= 
    \begin{cases}
    p^r \lrp{
        1 - \frac{1}{p}
        - \lrbb{ \frac{1}{p^2} }_{r\ge2}
        + \lrbb{ \frac{1}{p^3} }_{r\ge3}
    } & \text{if } p \nmid f, \\
    p^r \lrp{
        1 - \frac{2}{p}
        + \lrbb{ \frac{1}{p^2} }_{r\ge2}
    } & \text{if } p \mid f.
    \end{cases}
\end{align}
(Here, $\lrbb{A}_{\textit{condition}}$ denotes that the term $A$ only appears when $\textit{condition}$ is satisfied.) 

Now, observe that 
\begin{align}
    \lrp{
        1 - \frac{1}{p}
        - \lrbb{ \frac{1}{p^2} }_{r\ge2}
        + \lrbb{ \frac{1}{p^3} }_{r\ge3}
    } &\ge \frac14 \qquad \text{for all prime powers $p^r$}, \\
    \lrp{
        1 - \frac{2}{p}
        + \lrbb{ \frac{1}{p^2} }_{r\ge2}
    } &\ge \frac14 \qquad \text{for all prime powers $p^r \ne 2$}.
\end{align}
Hence assuming it is not the case that case that $2 \mid f,\ 2 \midmid N/f$, we have that
\begin{align}
    \psi_{f}^\nw(N) = \psi(f) \cdot \beta{*}\psi_f(N/f) \ge \psi(f) \frac{N/f}{4^{\omega(N/f)}} \ge f \frac{N/f}{4^{\omega(N)}} = \frac{N}{4^{\omega(N)}}, 
\end{align}
as desired.
\end{proof}

Then applying Lemma \ref{lemma:trace-estimate} at $m=1$ and $m=p^n$ (and using the fact that $4^{\omega(N)} = O(N^\varepsilon)$), we obtain that
\begin{align}
    \frac{\Tr \T^\nw_{p^n}(N,k,\chi) }{\Tr \T_1^\nw(N,k,\chi)} 
    = \frac{\mathbbm{1}_{2|n} p^{-n/2} 
        \frac{k-1}{12} \psi_{\condf(\chi)}^\nw(N) + O(N^{1/2+\varepsilon}) 
    }{
        \frac{k-1}{12} \psi_{\condf(\chi)}^\nw(N) + O(N^{1/2+\varepsilon})  
    }
    \longrightarrow \mathbbm{1}_{2|n} p^{-n/2} \quad \text{as} \quad N+k \to \infty.
\end{align}
This verifies \eqref{eqn:quot-traces}, completing the proof of Theorem \ref{thm:main-thm}.

\section{An Application of the Main Theorem} \label{sec:application}
Finally, we point out one application of Theorem \ref{thm:main-thm}, analogous to \cite[Theorem 5]{serre}.

For newforms $f \in S_k^\nw(N,\chi)$, let $K_f := \QQ(\{a_f(m) : m \ge 1, \ (m,N)=1\})$ denote the coefficient field of $f$. Then for $r \ge 1$, let $s(N,k,\chi)_r$ denote the number of newforms $f \in S_k^\nw(N,\chi)$ such that $[K_f:\QQ] = r$. Then the equidistribution result of Theorem \ref{thm:main-thm} implies the following corollary.

\begin{corollary}
    Fix $p$ prime, $k \ge 2$, and $r \ge 1$. Consider $N$ coprime to $p$ and Dirichlet characters $\chi$ modulo $N$ with $\chi(-1)=(-1)^k$. Then assuming it is not the case that $2 \mid \condf(\chi)$, $2 \midmid N/\condf(\chi)$; 
    \begin{align}
        \frac{s(N,k,\chi)_r}{\dim S_k^\nw(N,k,\chi)} \longrightarrow 0 \qquad \text{as} \qquad N \longrightarrow \infty.
    \end{align}
\end{corollary}

Recall that the Fourier coefficients $a_f(m)$ correspond with the eigenvalues of $T_m$.  
Then the above corollary follows from the fact that the set of possible $T_m$ eigenvalues with degree $\le r$ is finite, so the set of normalized eigenvalues is $\mu_p$-measure $0$ in $[-2,2]$. The interested reader can find the precise details of this proof in \cite[Theorem 5]{serre}.

\bibliographystyle{plain}
\bibliography{bibliography.bib}

@article {BGHT,
    AUTHOR = {Barnet-Lamb, Tom and Geraghty, David and Harris, Michael and
              Taylor, Richard},
     TITLE = {A family of {C}alabi-{Y}au varieties and potential automorphy
              {II}},
   JOURNAL = {Publ. Res. Inst. Math. Sci.},
  FJOURNAL = {Publications of the Research Institute for Mathematical
              Sciences},
    VOLUME = {47},
      YEAR = {2011},
    NUMBER = {1},
     PAGES = {29--98},
      ISSN = {0034-5318,1663-4926},
   MRCLASS = {11F80 (11F11 11G18 14J32)},
  MRNUMBER = {2827723},
MRREVIEWER = {Neil\ P.\ Dummigan},
       DOI = {10.2977/PRIMS/31},
       URL = {https://doi.org/10.2977/PRIMS/31},
}

@book {harmonic-maass-forms,
      AUTHOR = {Bringmann, Kathrin and Folsom, Amanda and Ono, Ken and Rolen,
              Larry},
     TITLE = {Harmonic {M}aass forms and mock modular forms: theory and
              applications},
    SERIES = {American Mathematical Society Colloquium Publications},
    VOLUME = {64},
 PUBLISHER = {American Mathematical Society, Providence, RI},
      YEAR = {2017},
     PAGES = {xv+391},
      ISBN = {978-1-4704-1944-8},
   MRCLASS = {11F03 (11F11 11F27 11F30 11F37 11F50 11P81 11P82)},
  MRNUMBER = {3729259},
MRREVIEWER = {Jeremy\ Lovejoy},
       DOI = {10.1090/coll/064},
       URL = {https://doi.org/10.1090/coll/064},
}

@article {serre,
    AUTHOR = {Serre, Jean-Pierre},
     TITLE = {R\'{e}partition asymptotique des valeurs propres de
              l'op\'{e}rateur de {H}ecke {$T_p$}},
   JOURNAL = {J. Amer. Math. Soc.},
  FJOURNAL = {Journal of the American Mathematical Society},
    VOLUME = {10},
      YEAR = {1997},
    NUMBER = {1},
     PAGES = {75--102},
      ISSN = {0894-0347,1088-6834},
   MRCLASS = {11F30 (11F25 11G20 11N37 11R45)},
  MRNUMBER = {1396897},
MRREVIEWER = {Glenn\ Stevens},
       DOI = {10.1090/S0894-0347-97-00220-8},
       URL = {https://doi.org/10.1090/S0894-0347-97-00220-8},
}

@article{ross,
title = {Newspaces with nebentypus: An explicit dimension formula and classification of trivial newspaces},
journal = {Journal of Number Theory},
volume = {278},
pages = {317-352},
year = {2026},
issn = {0022-314X},
doi = {https://doi.org/10.1016/j.jnt.2025.04.003},
url = {https://www.sciencedirect.com/science/article/pii/S0022314X25001477},
author = {Erick Ross},
}

@article{cason-et-al,
author = {Cason, William and Jim, Akash and Medlock, Charlie and Ross, Erick and Vilardi, Trevor and Xue, Hui},
title = {Nonvanishing of second coefficients of Hecke polynomials on the newspace},
journal = {International Journal of Number Theory},
volume = {21},
number = {07},
pages = {1479-1512},
year = {2025},
doi = {10.1142/S1793042125500757},
}

@book {cohen-stromberg,
    AUTHOR = {Cohen, Henri and Str\"{o}mberg, Fredrik},
     TITLE = {Modular forms: A classical approach},
    SERIES = {Graduate Studies in Mathematics},
    VOLUME = {179},
 PUBLISHER = {American Mathematical Society, Providence, RI},
      YEAR = {2017},
     PAGES = {xii+700},
      ISBN = {978-0-8218-4947-7},
   MRCLASS = {11-01 (11Fxx)},
  MRNUMBER = {3675870},
MRREVIEWER = {Sander\ Zwegers},
       DOI = {10.1090/gsm/179},
       URL = {https://doi.org/10.1090/gsm/179},
}

\end{document}